\documentclass[twoside,12pt]{article}
\usepackage{amsmath, amsfonts, amssymb, amsthm, graphics}

\newtheorem{theorem}{Theorem}[section]
\newtheorem{lemma}{Lemma}[section]
\newtheorem{proposition}{Proposition}[section]

\newtheorem{definition}{Definition}[section]
\newtheorem{remark}{Remark}[section]
\newtheorem{example}{Example}[section]

\catcode`@=11 \@addtoreset{equation}{section} \catcode`@=12

\begin{document}
\title{\bf Linear Time-Varying Dynamic-Algebraic Equations of Index One on Time Scales}

\author{Svetlin G. Georgiev\footnote{Sorbonne University, Department of Mathematics, Paris, France} and Sergey Kryzhevich\footnote{Gda\'nsk University of Technology, Faculty of Applied Physics and Mathematics, Gda\'nsk, Poland, serkryzh@pg.edu.pl}}
\maketitle
\begin{abstract}
In this paper, we introduce a class of linear time-varying dynamic-algebraic equations(LTVDAE) of tractability index one on arbitrary time scales. We propose a procedure for the decoupling of the considered class LTVDAE.  A projector approach is used to prove the main statement of the paper.
\end{abstract}


\section{Introduction}
The time scale dynamics, first introduced by Aulbach and Hilger \cite{{AH90}}, became an important part of the modern theory of dynamical systems. They appear in biological systems, mechanical systems, describing dynamics with impulse interactions, sociological models, etc. Roughly speaking, those systems combine features of ordinary differential equations and discrete dynamical systems. A brief survey on time scale calculus and time scale dynamics is performed in the book \cite{bohner1}, see also \cite{BG16}, \cite{D05}, \cite{M16}, references therein, and Section 2 of this paper.

In this paper, we study the problem of decoupling linear systems, investigating the LTVDAE
\begin{equation}
\label{1} A^{\sigma}(t)(Bx)^{\Delta}(t)=C^{\sigma}(t) x^{\sigma}(t)+f(t),\quad t\in I,
\end{equation}
where $A: I\to \mathcal{M}_{n\times m}$, $B: I\to \mathcal{M}_{m\times n}$, $C: I\to \mathcal{M}_{n\times n}$,  $f: I\to \mathbb{R}^n$ are given, $x: I\to \mathbb{R}^n$ is unknown,  $I\subseteq \mathbb{T}$, $\mathbb{T}$ is a time scale with forward jump operator and delta differentiation operator $\sigma$ and $\Delta$, respectively. Here, with $\mathcal{M}_{p\times q}$ we denote the space of $C^1$ -- smooth $p\times q$ matrices with real entries.  More precisely, we give conditions for $A$, $B$, $C$, and $f$ under which we construct projectors and matrix chains ensuring decoupling of the LTVDAE \eqref{1}. To our knowledge, no investigations were performed devoted to LTVDAE  on arbitrary time scales.

The paper is organized as follows.  In the next section, we give some basic facts from time scale calculus. In Section 3, we introduce $P$-projectors and matrix chains and we deduct some of their properties. In Section 4, we consider a particular case. In Section 5, we investigate standard form index one problems. In Section 6, we give a procedure for decoupling of the equation \eqref{1} and provide the main result of our paper. In Section 7, we give some examples, illustrating the obtained results. A conclusion is made in Section 8.

\section{Time Scales Essentials}

In this paper, a time scale denoted by the symbol $\mathbb T$, is any closed non-empty subset of $\mathbb R$.

We suppose that a time scale $\mathbb T$ has the topology that inherits from the real numbers with the standard topology. 

Also, in this section only, we assume for simplicity that either $\sup {\mathbb T}=+\infty$ or $\sup {\mathbb T}$ is not an isolated point.

We set $\inf\emptyset=\sup{\mathbb T}$ and $\sup\emptyset=\inf{\mathbb T}$.

\begin{definition}
For $t \in {\mathbb T}$ we define the forward jump operator $\sigma : {\mathbb T} \mapsto {\mathbb T}$ as follows
$\sigma(t) = \inf\{s \in {\mathbb T} : s > t\}$. We note that $\sigma(t) \ge t$ for any $t \in {\mathbb T}.$
\end{definition}
\begin{definition}
We define the graininess function by the formula $\mu(t)=\sigma(t)-t$. The point is called \emph{right-dense} if $\mu(t)=0$ and right-scattered otherwise. Left-dense and left-scattered points are defined similarly.
\end{definition}
\begin{definition} If $f : {\mathbb T} \mapsto {\mathbb R}$ is a function, then we define
$f^\sigma : {\mathbb T} \mapsto {\mathbb R}$ by
$f^{\sigma}(t)=f(\sigma(t))$ for any $t \in {\mathbb T}$, i.e., $f^\sigma =f\circ \sigma$.
\end{definition}

\begin{definition}
We say that a function is rd-continuous if it is continuous in all right-dense points and there exists a left limit at left-dense points (that may not coincide with the value of the function at the point).
\end{definition}

Here is the definition of a segment of a time scale.

\begin{definition}
Let $a,b\in {\mathbb T}$, $a\le b$. We define the interval $[a,b]$ in ${\mathbb T}$
by $$[a, b] = \{t \in {\mathbb T} : a \le t \le b\}.$$
\end{definition}

Open intervals, half-open intervals, half-lines and  so on are defined accordingly.

\begin{definition}
Assume that $f : {\mathbb T} \mapsto {\mathbb R}$ is a rd-continuous function and let $t \in {\mathbb T}$. We define $f^\Delta(t)$ to be the number, provided it exists, as follows: for any $\varepsilon > 0$ there is a neighborhood $U$ of $t$, $U=(t-\delta,t+\delta)\bigcap {\mathbb T}$ for some
$\delta > 0$, such that
$$|f(\sigma(t))-f(s)-f^\Delta(t)(\sigma(t)-s)|\le \varepsilon|\sigma(t)-s|\quad \mbox{for all}\quad s\in U, s \neq \sigma(t).$$
We say that $f^\Delta(t)$ is the delta or Hilger derivative of $f$ at $t$.

We say that $f$ is delta or Hilger differentiable, shortly differentiable, in $\mathbb T$ if $f^\Delta (t)$ exists for all $t \in {\mathbb T}$. The function $f^\Delta : {\mathbb T} \mapsto {\mathbb R}$ is said to be delta derivative or Hilger derivative, shortly derivative, of $f$ in
$\mathbb T$.
\end{definition}

If ${\mathbb T}={\mathbb R}$, then the delta derivative coincides with the classical derivative, if ${\mathbb T} = {\mathbb N}$, this is just the increment $f(n+1)-f(n)$.

We list some basic properties of Hilger derivatives.

\begin{theorem} Let $f :{\mathbb T} \mapsto R$ is a function and let $t\in {\mathbb T}$. Then the
following holds.
\begin{enumerate}
\item If $f$ is differentiable at $t$, then $f$ is continuous at $t$.
\item If $f$ is rd-continuous and $t$ is right-scattered, then $f$ is differentiable at $t$ with
$$f^\Delta(t)= \dfrac{f(\sigma(t))-f(t)}{\mu(t)}.$$
\item If $t$ is right-dense, then $f$ is differentiable if and only if the limit
$$f^\Delta(t):=\lim_{s\to t} \dfrac{f(t)-f(s)}{t-s}$$
exists and is finite.
\item If $f$ is differentiable at $t$, then
$f (\sigma (t)) = f (t) + \mu(t)f^\Delta (t)$.
\end{enumerate}
\end{theorem}

We list some basic properties of the Hilger derivatives.

\begin{theorem} Assume that $f,g:{\mathbb T}\mapsto {\mathbb R}$ are differentiable at $t\in {\mathbb T}$. Then
\begin{enumerate}
\item the sum $f+g$ is differentiable at $t$ with
$$(f+g)^\Delta(t)=f^\Delta(t)+g^\Delta(t).$$
\item for any constant $\alpha$, the function $\alpha f$ is differentiable at $t$ with $$(\alpha f )^\Delta(t) = \alpha f^\Delta(t).$$
\item if $g(t),g(\sigma(t))\neq 0$ then $f/g$ is differentiable at $t$ with
$$(f/g)^\Delta (t)= \dfrac{f^\Delta(t)g(t)-f(t)g^\Delta(t)}{g(t)g(\sigma(t))}.$$
\item the product $fg$ is differentiable at $t$ with
$$(fg)^\Delta(t) = f^\Delta(t)g(t)+f(\sigma(t))g^\Delta(t)
= f(t)g^\Delta(t)+f^\Delta(t)g(\sigma(t)).$$
\end{enumerate}
\end{theorem}

The integral calculus on time scales is also quite well-developed, see \cite{bohner1}. Now we introduce some basic concepts related to the theory of linear systems of the form
\begin{equation}\label{linear}
    x^\Delta=A(t)x
\end{equation}
on time scales.

Properties of solutions of system \eqref{linear} are also studied. In general, they correspond to those for ordinary differential equations with one important exception which can be illustrated by the following example.

\noindent\textbf{Example}. Let ${\mathbb T}={\mathbb N}$. Consider the scalar equation
$x^\Delta=-x$. Then any solution of the considered equation with initial conditions $x(n_0)=x_0$ is zero for any $n>n_0$, so there is no backward uniqueness of solutions.

\begin{definition}
We say that the matrix $A$ is regressive with respect to $\mathbb T$ provided $E+\mu(t)A(t)$ is invertible for all $t\in {\mathbb T}$.
\end{definition}

\begin{theorem}
\begin{enumerate}
    \item Any solution of system \eqref{linear} with initial conditions $x(t_0)=x_0$ is correctly defined for any $t\ge t_0$.
    \item Solutions of system \eqref{linear} are unique if the matrix is regressive.
\item The matrix-valued function $A$ is regressive if and only if the eigenvalues $\lambda_i(t)$ of $A(t)$ are regressive for all $1 \le  i \le  n$.
\end{enumerate}
\end{theorem}

An analog of the Lagrange method to solve linear non-homogeneous systems holds true for time scale systems, as well.

\section{$P$-Projectors. Matrix Chains}
Below, without loss of generality, we remove the dependence on $t$ for the sake of notational simplicity.
\subsection{\emph{Properly Stated Matrix Pairs}}
Suppose that $A: I\to \mathcal{M}_{n\times m}$ and  $B: I\to \mathcal{M}_{m\times n}$.
\begin{definition}
We will say that the matrix pair $(A, B)$ is properly stated on $I$ if $A$ and $B$ satisfy
\begin{equation}
\label{3.1} \text{ker}\,A \oplus \text{im}\, B=\mathbb{R}^m
\end{equation}
 and both $\text{ker}\,A$ and $\text{im}\,B$ are $\mathcal{C}^1$-spaces. The condition \eqref{3.1} is said to be the transversality condition.
\end{definition}
Let $(A, B)$ be a properly stated matrix pair on $I$.  Then $A$, $B$, and $AB$ have a constant rank on $I$ and
\begin{equation*}
\label{3.2} \text{ker}\,AB=\text{ker}\,B.
\end{equation*}
Indeed, \eqref{3.2} implies that if $y\in \text{im}\, B$ then $Ay\neq 0$.

Note that, by the transversality condition \eqref{3.1} together with the $\mathcal{C}^1$-smoothness of $\text{ker}\,A$ and $\text{im}\,B$, it follows that there is a projector $R\in \mathcal{C}^1(I)$ onto $\text{im}\,B$ and along $\text{ker}\,A$ on $I$.
\begin{theorem}
\label{theorem3.4} Let $N=\text{ker}\,A$ and  $\text{dim}\,N=m-r$. Then there exists a  smooth projector $Q$  onto $N$ on $I$ if and only if $N$ can be spanned by $\{n_1, \ldots, n_{m-r}\}$ on $I$, where $n_j\in \mathcal{C}^1(I)$, $j\in \{1, \ldots, m-r\}$.
\end{theorem}
\begin{proof}

\begin{enumerate}
\item Suppose that there is a smooth projector $Q$ onto $N$. Take $t_0\in I$ arbitrarily. Let
\begin{equation*}
n_1^0,\ldots,n_{m-r}^0\in \mathbb{R}^m
\end{equation*}
be a basis of $N(t_0)$.  We choose $n_j$, $j\in \{1, \ldots, m-r\}$,  as the solution of the initial value
\begin{eqnarray*}
n_j^{\Delta}&=& Q^{\Delta}n_j\quad \text{on}\quad I,\\ \\
n_j(t_0)&=& n_j^0,\quad j\in \{1, \ldots, m-r\}.
\end{eqnarray*}
Let also,
\begin{equation*}
P=E-Q.
\end{equation*}
Here $E$ stands for the unit matrix. Then
\begin{equation*}
PQ=0,\quad Pn_j^0=0,\quad j\in \{1, \ldots, m-r\},
\end{equation*}
and
\begin{eqnarray*}
0&=& (PQ)^{\Delta}= P^{\Delta}Q+P^{\sigma}Q^{\Delta},
\end{eqnarray*}
whereupon
\begin{equation*}
P^{\Delta}Q=-P^{\sigma}Q^{\Delta}.
\end{equation*}
Next,
\begin{equation*}
P^{\sigma}n_j^{\Delta}= P^{\sigma}Q^{\Delta} n_j,
\end{equation*}
and
\begin{eqnarray*}
(Pn_j)^{\Delta}&=& P^{\Delta}n_j+P^{\sigma}n_j^{\Delta}= P^{\Delta}n_j+P^{\sigma}Q^{\Delta}n_j
= P^{\Delta}n_j-P^{\Delta}Q n_j\\ \\
&=& P^{\Delta}(E-Q)n_j= P^{\Delta}Pn_j\quad \text{on}\quad I, \quad j\in \{1, \ldots, m-r\}.
\end{eqnarray*}
So, we obtain the following IVP
\begin{eqnarray*}
(Pn_j)^{\Delta}&=& P^{\Delta}Pn_j\quad \text{on}\quad I\\ \\
Pn_j(t_0)&=& 0,\quad j\in \{1, \ldots, m-r\}.
\end{eqnarray*}
Therefore $Pn_j=0$ on $I$, $j\in \{1, \ldots, m-r\}$,
 and $n_j\in N$ on  $I$. Since $n_1, \ldots, n_{m-r}$ are linearly independent  on $I$ (otherwise, the dimension of $N$ is non-constant), we conclude that they span $N$ on $I$.
\item Let $N$ be spanned by
\begin{equation*}
F=(n_1, \ldots, n_{m-r}),\quad n_j\in \mathcal{C}^1(I),\quad j\in \{1, \ldots, m-r\},
\end{equation*}
on $I$. Take
\begin{equation*}
Q= F\left(F^TF\right)^{-1}F^T\quad\text{on}\quad  I.
\end{equation*}
Observe that the matrix $F^TF$ is invertible since vectors $n_1,\ldots, n_{m-r}$ are linearly independent. Then
\begin{eqnarray*}
QQ&=& F\left(F^TF\right)^{-1}F^TF\left(F^TF\right)^{-1}F^T\\ \\
&=& F\left(F^TF\right)^{-1}F^T\\ \\
&=& Q.
\end{eqnarray*}
Note that $Q\in \mathcal{C}^1(I)$. This completes the proof.
\end{enumerate}
\end{proof}
\begin{definition}
If the null space $N=\text{ker}\,A$, $\text{dim}\,N=m-r$, satisfies Theorem \ref{theorem3.4}, we will say that $N$ depends on $t\in I$ smoothly.
\end{definition}

\subsection{\emph{Matrix Chains}}
Suppose that $(A, B)$ is a properly stated matrix pair  on $I$ and  $C\in \mathcal{C}(I)$. Set
\begin{equation*}
G_0=AB.
\end{equation*}
Denote with $P_0$ a projector along $\text{ker}\,G_0$. Assume that $R$ is a continuous projector onto $\text{im}\,B$ and along $\text{ker}\,A$. With $B^{-}$ we denote the \emph{$\{1, 2\}$-inverse} of $B$,  determined by
\begin{equation}
\label{3.5}
\begin{array}{lll}
BB^-B&=& B,\quad
B^-BB^-= B^-,\\ \\
B^-B&=& P_0,\quad
BB^-= R.
\end{array}
\end{equation}
Note that the matrix $B^-$ is uniquely determined by the conditions \eqref{3.5}. Set
\begin{equation*}
Q_0=E-P_0.
\end{equation*}
We assume that a matrix $G_1: I \mapsto {\mathcal M}_{m,m}$ is such that
\begin{description}
\item[(A1)] $G_1$  has a constant rank $r_1$ on $I$.
\item[(A2)] $N_1=\text{ker}\,G_1$ satisfies $N_0\cap N_1=\{0\}$.
\end{description}
We choose a continuous projector $Q_1$ onto $N_1$ such that
\begin{description}
\item[(A3)] $Q_1Q_0=0$.
\end{description}
Set $P_1=E-Q_1$ and assume that
\begin{description}
\item[(A4)] $BP_0 P_1 B^-\in \mathcal{C}^1(I)$.
\end{description}
Assume that $(A1)$-$(A4)$ hold and set $C_0=C$.
Define
\begin{eqnarray*}
G_1&=& G_0+C_0 Q_0.
\end{eqnarray*}
\begin{remark}
Note that the existence of a projector $Q_1$ so that $Q_1 Q_0=0$ relies on the fact that the condition $(A2)$ makes it possible to choose a projector $Q_1$ onto $N_1$ such that
$N_0\subseteq \text{ker}\,Q_1$.
Since $Q_0$ projects onto $N_0$,  we have
$\text{im}\,Q_0=N_0$.
Hence, for any $x\in \mathbb{R}^m$, we have
\begin{equation*}
Q_0 x\in N_0\subseteq \text{ker}\,Q_1
\end{equation*}
and then $Q_1 Q_0 x=0$, i.e., $Q_1Q_0=0$.
\end{remark}
\begin{definition}\label{admissible}
The projectors  $\{Q_0,  Q_1\}$, respectively $\{P_0,  P_1\}$,  are said to be admissible up to level $1$ if $(A1)$, $(A2)$  and $(A3)$ hold.
\end{definition}
\begin{definition}
The matrix pair $(A, B)$ is said to be the leading term for the LTVDAE \eqref{1}.
\end{definition}
\begin{definition}
Let the leading term of \eqref{1} is properly stated on $I$. The equation \eqref{1} is said to be regular with tractability index $1$ if there are projectors $Q_0$ and $Q_1$ that are admissible up to level $1$, $G_0$ is singular and $G_1$ is nonsingular.
\end{definition}
\begin{proposition}\label{proposition3.6}
Let $\{Q_0, Q_1\}$ be admissible up to level $1$. Then $\text{ker}\,P_0P_1=N_0\oplus N_1$.
\end{proposition}

\begin{proof}

\begin{enumerate}
\item Suppose that $z\in \text{ker}\,P_0 P_1$. Then
\begin{equation}
\label{3.7} P_0P_1z=0.
\end{equation}
Set
\begin{eqnarray*}
z_1&=& P_1z= (E-Q_1)z= z-Q_1z,
\end{eqnarray*}
i.e.,
\begin{equation*}
z=z_1+Q_1 z.
\end{equation*}
Hence, \eqref{3.7} holds if and only if $z_1\in \text{ker}\,P_0=N_0$. Note that $Q_1z\in N_1$. Therefore
\begin{equation*}
z\in N_0\oplus N_1.
\end{equation*}
Because $z\in \text{ker}\,P_0P_1$ was arbitrarily chosen and we get that it is an element of $N_0\oplus N_1$, we arrive at the relation
\begin{equation}
\label{3.8} \text{ker}\,P_0 P_1\subseteq N_0\oplus N_1.
\end{equation}
\item Let $z\in N_0\oplus N_1$ be arbitrarily chosen. Then
\begin{equation*}
z=z_0+z_1,
\end{equation*}
where $z_0\in N_0$ and $z_1\in N_1$. Since
\begin{equation*}
\text{ker}\,P_1=N_1,
\end{equation*}
we have
\begin{equation*}
P_1z_1=0.
\end{equation*}
Next, by $z_0\in N_0$, it follows that there is a $w_0\in \mathbb{R}^m$ so that
\begin{equation*}
z_0=Q_0 w_0.
\end{equation*}
Therefore
\begin{eqnarray*}
P_0P_1z&=& P_0P_1(z_0+z_1)= P_0P_1z_0+P_0P_1z_1= P_0P_1Q_0w_0\\ \\
&=& P_0(E-Q_1)Q_0 w_0= P_0(Q_0-Q_1Q_0)w_0= P_0 Q_0 w_0= 0,
\end{eqnarray*}
i.e., $z\in \text{ker}\,P_0P_1$. Since $z\in N_0\oplus N_1$ was arbitrarily chosen and we get that it is an element of $\text{ker}\, P_0P_1$, we obtain the relation
\begin{equation*}
N_0\oplus N_1\subseteq \text{ker}\,P_0P_1.
\end{equation*}
By the last inclusion and \eqref{3.8}, we obtain
$\text{ker}\,P_0P_1=N_0\oplus N_1$.
This completes the proof.
\end{enumerate}
\end{proof}
\subsection{\emph{Independence of the Matrix Chains}}
We will start this subsection with the following useful lemmas.
\begin{lemma}
\label{lemma3.10} Let $\{Q_0, Q_1\}$ and $\{\overline{Q}_0, \overline{Q}_1\}$ be two admissible up to level $1$ projector sequences and $\{G_0, N_0,  G_1, N_1\}$ and $\{\overline{G}_0, \overline{N}_0,   \overline{G}_1, \overline{N}_1\}$ be the corresponding sequences. Then
\begin{equation*}
\overline {B}^-= \overline{P}_0B^-,
\end{equation*}
\begin{equation}
\label{3.12}
Q_0 \overline{Q}_0=\overline{Q}_0,
\end{equation}
\begin{equation}
\label{3.13} \overline{Q}_0 Q_0=Q_0,
\end{equation}
\begin{equation*}
 P_0 \overline{P}_0=P_0,
\end{equation*}
\begin{equation*}
\overline{P}_0P_0=\overline{P}_0
\end{equation*}
and
\begin{equation*}
G_0 Q_0=G_0\overline{Q}_0=0.
\end{equation*}
\end{lemma}
\begin{proof}
We have by definition
\begin{eqnarray*}
G_0&=& \overline{G}_0,\quad N_0= \overline{N}_0
\end{eqnarray*}
and
\begin{equation*}
\text{ker}\,P_0=\text{ker}\,\overline{P}_0.
\end{equation*}
Moreover,
\begin{eqnarray*}
BB^-B&=& B,\quad B^-BB^-= B^-,\quad BB^-= R,\quad B^-B= P_0
\end{eqnarray*}
and
\begin{eqnarray*}
B\overline{B}^-B&=& B,\quad
\overline{B}^-B \overline{B}^-= \overline{B}^-,\quad B\overline{B}^-= R,\quad \overline{B}^-B= \overline{P}_0.
\end{eqnarray*}
Hence,
\begin{eqnarray*}
\overline{P}_0B^-&=& \overline{B}^-BB^-= \overline{B}^-R= \overline{B}^-B \overline{B}^-= \overline{B}^-.
\end{eqnarray*}
Since $N_0=\overline{N}_0$, $\text{im}\,Q_0=N_0$, $\text{im}\,\overline{Q}_0=\overline{N}_0$, we have
\begin{equation*}
\text{im}\,Q_0=\text{im}\,\overline{Q}_0.
\end{equation*}
Take $z\in \mathbb{R}^m$ arbitrarily. If $z\in \text{ker}\,\overline{Q}_0$, then
\begin{equation*}
\overline{Q}_0z=0
\end{equation*}
and
\begin{eqnarray*}
Q_0 \overline{Q}_0 z&=& Q_00= 0.
\end{eqnarray*}
If $z\in \text{im}\,\overline{Q}_0$, then since $\overline{Q}_0$ is a projector, we have
\begin{equation*}
\overline{Q}_0 z=z
\end{equation*}
and
\begin{equation*}
z\in \text{im}\,Q_0=\text{im}\,\overline{Q}_0.
\end{equation*}
Therefore
\begin{equation*}
Q_0 z=z
\end{equation*}
and
\begin{eqnarray*}
Q_0 \overline{Q}_0z&=& Q_0 z= z= \overline{Q}_0z.
\end{eqnarray*}
Because $z\in \mathbb{R}^m$ was arbitrarily chosen, we conclude that
\begin{equation*}
Q_0 \overline{Q}_0=\overline{Q}_0.
\end{equation*}
Let $y\in \mathbb{R}^m$ be arbitrarily chosen. If $y\in \text{ker}\,Q_0$, then
\begin{equation*}
Q_0y=0
\end{equation*}
and
\begin{eqnarray*}
\overline{Q}_0Q_0y&=& \overline{Q}_0= 0= Q_0y.
\end{eqnarray*}
If $y\in \text{im}\,Q_0$, then $y\in \text{im}\,\overline{Q}_0$ and
\begin{eqnarray*}
\overline{Q}_0y&=& Q_0y= y.
\end{eqnarray*}
From here,
\begin{eqnarray*}
\overline{Q}_0 Q_0y&=& \overline{Q}_0 y= y= Q_0y.
\end{eqnarray*}
Because $y\in \mathbb{R}^m$ was arbitrarily chosen, we arrive at
\begin{equation*}
\overline{Q}_0Q_0=Q_0.
\end{equation*}
By \eqref{3.12}, we find
\begin{equation*}
E-\overline{P}_0=(E-P_0)(E-\overline{P}_0)= E-P_0-\overline{P}_0+P_0\overline{P}_0,
\end{equation*}
whereupon
\begin{equation*}
P_0\overline{P}_0=P_0.
\end{equation*}
By \eqref{3.13}, we find
\begin{eqnarray*}
E-P_0&=& (E-\overline{P}_0)(E-P_0)= E-P_0-\overline{P}_0+\overline{P}_0P_0,
\end{eqnarray*}
from where
\begin{equation*}
\overline{P}_0P_0=\overline{P}_0.
\end{equation*}
Since
\begin{equation*}
\text{im}\,Q_0=N_0=\text{ker}\,G_0=\text{im}\,\overline{Q}_0,
\end{equation*}
we get
\begin{eqnarray*}
G_0Q_0&=& G_0 \overline{Q}_0= 0.
\end{eqnarray*}
This completes the proof.
\end{proof}
\begin{lemma}
\label{lemma3.17} Let $\{Q_0,  Q_1\}$ and $\{\overline{Q}_0,  \overline{Q}_1\}$ be two admissible up to level $1$ projector sequences and $\{G_0, N_0,   G_1, N_1\}$ and $\{\overline{G}_0, \overline{N}_0,   \overline{G}_1, \overline{N}_1\}$ be the corresponding sequences. If
\begin{equation*}
Z_1=E+Q_0 \overline{Q}_0P_0,
\end{equation*}
then
\begin{equation*}
\overline{G}_1=G_1Z_1
\end{equation*}
and
\begin{equation*}
\text{ker}\,P_0P_1=\text{ker}\,\overline{P}_0\overline{P}_1=N_0\oplus N_1=\overline{N}_0\oplus \overline{N}_1.
\end{equation*}
\end{lemma}
\begin{proof}
Applying \eqref{3.12} and \eqref{3.13}, we get
\begin{eqnarray*}
\overline{G}_1&=&  G_0+C \overline{Q}_0= G_0+C Q_0 \overline{Q}_0= G_0+CQ_0\overline{Q}_0(P_0+Q_0)\\ \\
&=& G_0+C Q_0\overline{Q}_0P_0+C Q_0\overline{Q}_0Q_0= G_0+C Q_0\overline{Q}_0P_0+C Q_0Q_0= G_0+C Q_0+C Q_0\overline{Q}_0P_0\\ \\
&=& G_1+C Q_0\overline{Q}_0P_0= G_1+(G_0 Q_0+CQ_0^2)\overline{Q}_0 P_0= G_1+(G_0+C Q_0)Q_0 \overline{Q}_0P_0\\ \\
&=& G_1+G_1Q_0 \overline{Q}_0P_0= G_1(E+Q_0 \overline{Q}_0P_0)= G_1Z_1.
\end{eqnarray*}
Hence,
\begin{equation*}
N_1=Z_1^{-1}\overline{N}_1
\end{equation*}
and thus,
\begin{equation*}
N_0\oplus N_1=\overline{N}_0\oplus \overline{N}_1.
\end{equation*}
By Proposition \ref{proposition3.6}, it follows that
\begin{equation*}
\text{ker}\,P_0 P_1=\text{ker}\,\overline{P}_0\overline{P}_1=N_0\oplus N_1=\overline{N}_0\oplus \overline{N}_1.
\end{equation*}
This completes the proof.
\end{proof}
By Lemma \ref{lemma3.10} and Lemma \ref{lemma3.17}, it follows the following result.
\begin{theorem}
The subspaces $\text{im}\,G_0$, $\text{im}\,G_1$, $N_0$ and $N_0\oplus N_1$ do not all depend on the special choice of the projectors $Q_0$ and $Q_1$.
\end{theorem}
\section{An Important Particular Case}
Consider the equation
\begin{equation}
\label{4.2} A^{\sigma}x^{\Delta}=C^{\sigma}x^{\sigma}+f.
\end{equation}
We will show that the equation \eqref{4.2} can be reduced to the equation \eqref{1}. Suppose that $P$ is a $\mathcal{C}^1$-projector along $\text{ker}\,A^{\sigma}$. Then
\begin{equation*}
A^{\sigma}P=A^{\sigma}
\end{equation*}
and
\begin{eqnarray*}
A^{\sigma}x^{\Delta}&=& A^{\sigma}Px^{\Delta}= A^{\sigma}(Px)^{\Delta}-A^{\sigma}P^{\Delta}x^{\sigma}.
\end{eqnarray*}
Hence, the equation \eqref{4.2} takes the form
\begin{eqnarray*}
A^{\sigma}(Px)^{\Delta}-A^{\sigma}P^{\Delta}x^{\sigma}=C^{\sigma}x^{\sigma}+f
\end{eqnarray*}
or
\begin{equation*}
A^{\sigma}(Px)^{\Delta}=(A^{\sigma}P^{\Delta}+C^{\sigma})x^{\sigma}+f.
\end{equation*}
Set
\begin{equation*}
C_1^{\sigma}=A^{\sigma}P^{\Delta}+C^{\sigma}.
\end{equation*}
Thus, \eqref{4.2} takes the form
\begin{equation}
\label{4.3} A^{\sigma}(Px)^{\Delta}=C^{\sigma}x^{\sigma}+f,
\end{equation}
i.e.,  the equation \eqref{4.2} is a particular case of the equation \eqref{1}.

\section{Standard Form Index One Problems}
In this section, we will investigate equation \eqref{4.3}
where $\text{ker}\,A$ is a $\mathcal{C}^1$-space, $C\in \mathcal{C}(I)$, $P$ is a $\mathcal{C}^1$-projector along $\text{ker}\,A$. Then
\begin{equation*}
AP=A.
\end{equation*}
Assume in addition that
\begin{equation*}
Q=E-P
\end{equation*}
and
\begin{description}
\item[(B1)] the matrix
\begin{equation*}
A_1=A+CQ
\end{equation*}
is invertible.
\end{description}
The condition $(B1)$ ensures the equation \eqref{4.3} to be regular with tractability index $1$. We will start our investigations with the following useful lemma.
\begin{lemma}
\label{lemma4.4} Suppose that $(B1)$ holds. Then
\begin{equation*}
A_1^{-1}A=P
\end{equation*}
and
\begin{equation*}
A_1^{-1}CQ=Q.
\end{equation*}
\end{lemma}
\begin{proof}
We have
\begin{eqnarray*}
A_1P&=& (A+CQ)P= AP+CQP= A.
\end{eqnarray*}
Since $Q=E-P$ and $\text{ker}\,P=\text{ker}\,A$, we have
$\text{im}\,Q=\text{ker}\,A$ and
\begin{equation*}
AQ=0.
\end{equation*}
Then
\begin{eqnarray*}
A_1Q&=& (A+CQ)Q= AQ+CQQ= CQ.
\end{eqnarray*}
This completes the proof.
\end{proof}
Now, we multiply the equation \eqref{4.3} with $\left(A_1^{-1}\right)^{\sigma}$ and we get
\begin{equation*}
\left(A_1^{-1}\right)^{\sigma}A^{\sigma}(Px)^{\Delta}=\left(A_1^{-1}\right)^{\sigma}C^{\sigma}x^{\sigma}+\left(A_1^{-1}\right)^{\sigma}f.
\end{equation*}
Now, we employ the first equation of Lemma \ref{lemma4.4} and we get
\begin{equation}
\label{4.5} P^{\sigma}(Px)^{\Delta}=\left(A_1^{-1}\right)^{\sigma}C^{\sigma}x^{\sigma}+\left(A_1^{-1}\right)^{\sigma}f.
\end{equation}
We decompose $x$ in the following way
\begin{equation*}
x=Px+Qx.
\end{equation*}
Then the equation \eqref{4.5} takes the following form
\begin{eqnarray*}
P^{\sigma}(Px)^{\Delta}&=& \left(A_1^{-1}\right)^{\sigma}C^{\sigma}(P^{\sigma}x^{\sigma}+Q^{\sigma}x^{\sigma})+\left(A_1^{-1}\right)^{\sigma}f\\ \\
&=& \left(A_1^{-1}\right)^{\sigma}C^{\sigma}P^{\sigma}x^{\sigma}+\left(A_1^{-1}\right)^{\sigma}C^{\sigma}Q^{\sigma}x^{\sigma}+\left(A_1^{-1}\right)^{\sigma}f.
\end{eqnarray*}
Using the second equation of Lemma \ref{lemma4.4},  the last equation can be rewritten as follows
\begin{equation}
\label{4.6} P^{\sigma}(Px)^{\Delta}=\left(A_1^{-1}\right)^{\sigma}C^{\sigma}P^{\sigma}x^{\sigma}+Q^{\sigma}x^{\sigma}+\left(A_1^{-1}\right)^{\sigma}f.
\end{equation}
We multiply the equation \eqref{4.6} with the projector $P^{\sigma}$ and  using that
\begin{eqnarray*}
PP&=& P,\quad PQ= 0,
\end{eqnarray*}
we find
\begin{equation*}
P^{\sigma}P^{\sigma}(Px)^{\Delta}=P^{\sigma}\left(A_1^{-1}\right)^{\sigma}C^{\sigma}P^{\sigma}x^{\sigma}+P^{\sigma}Q^{\sigma} x^{\sigma}+P^{\sigma}\left(A_1^{-1}\right)^{\sigma}f
\end{equation*}
or
\begin{equation}
\label{4.7} P^{\sigma}(Px)^{\Delta}=P^{\sigma}\left(A_1^{-1}\right)^{\sigma}C^{\sigma}P^{\sigma}x^{\sigma}+P^{\sigma}\left(A_1^{-1}\right)^{\sigma}f.
\end{equation}
Note that
\begin{eqnarray*}
P^{\sigma}(Px)^{\Delta}&=& (PPx)^{\Delta}-P^{\Delta}Px= (Px)^{\Delta}-P^{\Delta}P x.
\end{eqnarray*}
Hence and using \eqref{4.7}, we find
\begin{equation*}
(Px)^{\Delta}-P^{\Delta}Px=P^{\sigma}\left(A_1^{-1}\right)^{\sigma}C^{\sigma}P^{\sigma}x^{\sigma}+P^{\sigma}\left(A_1^{-1}\right)^{\sigma}f
\end{equation*}
or
\begin{equation}
\label{4.8} (Px)^{\Delta}=P^{\Delta}Px+P^{\sigma}\left(A_1^{-1}\right)^{\sigma}C^{\sigma}P^{\sigma}x^{\sigma}+P^{\sigma}\left(A_1^{-1}\right)^{\sigma}f.
\end{equation}
Now, we multiply the equation \eqref{4.6} by $Q^{\sigma}$ and we find
\begin{equation*}
Q^{\sigma}P^{\sigma}(Px)^{\Delta}=Q^{\sigma}\left(A_1^{-1}\right)C^{\sigma}P^{\sigma}x^{\sigma}+Q^{\sigma}Q^{\sigma}x^{\sigma}+
Q^{\sigma}\left(A_1^{-1}\right)^{\sigma}f
\end{equation*}
or
\begin{equation}
\label{4.9} 0= Q^{\sigma}\left(A_1^{-1}\right)^{\sigma}C^{\sigma}P^{\sigma}x^{\sigma}+Q^{\sigma}x^{\sigma}+Q^{\sigma}\left(A_1^{-1}\right)^{\sigma}f.
\end{equation}
Set
\begin{eqnarray*}
u&=& Px,\quad v= Qx.
\end{eqnarray*}
Then, by \eqref{4.8} and \eqref{4.9}, we get the system
\begin{equation}
\label{4.10}
\begin{array}{lll}
u^{\Delta}&=&P^{\Delta}u+P^{\sigma}\left(A_1^{-1}\right)^{\sigma}C^{\sigma}u^{\sigma}+P^{\sigma}\left(A_1^{-1}\right)^{\sigma}f\\ \\
v^{\sigma}&=& -Q^{\sigma}\left(A_1^{-1}\right)^{\sigma}C^{\sigma}u^{\sigma}-Q^{\sigma}\left(A_1^{-1}\right)^{\sigma}f.
\end{array}
\end{equation}
We find $u\in \mathcal{C}^1(I)$ by the first equation of the system \eqref{4.10} and then we  find $v^{\sigma}\in \mathcal{C}(I)$ by the second equation of the system \eqref{4.10}. Then, for  the solution $x$ of the equation \eqref{4.3}  we have the following representation
\begin{eqnarray*}
x^{\sigma}&=& u^{\sigma}+v^{\sigma}= P^{\sigma}x^{\sigma}+Q^{\sigma}x^{\sigma}.
\end{eqnarray*}
\section{Decoupling}
Suppose that the equation \eqref{1} is regular with tractability index $1$. In addition, assume that $R$ is a continuous projector onto $\text{im}\,B$ and along $\text{ker}\,A$. Set
\begin{equation*}
G_0=AB
\end{equation*}
and take $P_0$ to be a continuous projector along $\ker{G}_0$  and denote
\begin{eqnarray*}
Q_0&=& E-P_0,\quad G_1= G_0+C Q_0,\\ \\
BB^-B&=& B,\quad B^- BB^-= B^-,\quad B^-B= P_0,\quad BB^-= R,\quad N_0= \text{ker}\,G_0.
\end{eqnarray*}
Since
\begin{equation*}
R=BB^-
\end{equation*}
and $R$ is a continuous projector along $\text{ker}\,A$, we have
\begin{eqnarray*}
A&=& AR= ABB^{-}= G_0 B^-.
\end{eqnarray*}
Then, we can rewrite the equation \eqref{1} as follows
\begin{equation*}
A^{\sigma}B^{\sigma}B^{-\sigma}(Bx)^{\Delta}=C^{\sigma}x^{\sigma}+f
\end{equation*}
or
\begin{equation*}
G_0^{\sigma}B^{-\sigma}(Bx)^{\Delta}=C^{\sigma}x^{\sigma}+f.
\end{equation*}
Now, we multiply both sides of the last equation with $\left(G_{1}^{-1}\right)^{\sigma}$ and we find
\begin{equation}
\label{4.20} \left(G_{1}^{-1}\right)^{\sigma}G_0^{\sigma}B^{-\sigma}(Bx)^{\Delta}= \left(G_{1}^{-1}\right)^{\sigma}C^{\sigma}x^{\sigma}+\left(G_{1}^{-1}\right)^{\sigma}f.
\end{equation}
Note that by the definition of $G_0$ and $Q_0$ we have $G_0Q_0=0$. Then
\begin{eqnarray*}
G_1(E-Q_0)&=& (G_0+C Q_0)(E-Q_0)=G_0+C Q_0-G_0 Q_0-CQ_0 Q_0\\ \\
&=& G_0+C Q_0-CQ_0=G_0,
\end{eqnarray*}
whereupon
\begin{equation*}
G_1^{-1}G_0=E-Q_0.
\end{equation*}
Thus, the equation \eqref{4.20} takes the form
\begin{equation}
\label{4.21} (E-Q_0^{\sigma})B^{-\sigma}(Bx)^{\Delta}= \left(G_1^{-1}\right)^{\sigma}C^{\sigma}x^{\sigma}+\left(G_1^{-1}\right)^{\sigma}f.
\end{equation}
Note that
\begin{equation*}
B^{\sigma}P_0^{\sigma}(E-Q_0^{\sigma})=B^{\sigma}(P_0^{\sigma}-P_0^{\sigma}Q_0^{\sigma})=B^{\sigma}P_0^{\sigma}.
\end{equation*}
Then, we multiply the equation \eqref{4.21} by $B^{\sigma}P_0^{\sigma}$ and we find
\begin{equation}
\label{4.22} B^{\sigma}P_0^{\sigma} B^{-\sigma}(Bx)^{\Delta}=B^{\sigma}P_0^{\sigma} \left(G_1^{-1}\right)^{\sigma}C^{\sigma}x^{\sigma}+B^{\sigma}P_0^{\sigma}\left(G_1^{-1}\right)^{\sigma}f.
\end{equation}
Note that
\begin{eqnarray*}
B^{\sigma}P_0^{\sigma}B^{-\sigma}(Bx)^{\Delta}&=& (BP_0 B^-Bx)^{\Delta}-(BP_0 B^-)^{\Delta}Bx\\ \\
&=& (BP_0P_0x)^{\Delta}-(BP_0 B^-)^{\Delta}Bx= (BP_0x)^{\Delta}-(BP_0 B^-)^{\Delta} Bx.
\end{eqnarray*}
From here, the equation \eqref{4.22} can be rewritten as follows
\begin{equation*}
(BP_0x)^{\Delta}=(BP_0 B^-)^{\Delta}Bx+B^{\sigma}P_0^{\sigma}\left(G_1^{-1}\right)^{\sigma}C^{\sigma}x^{\sigma}+B^{\sigma}P_0^{\sigma}\left(G_1^{-1}\right)^{\sigma}f.
\end{equation*}
Now, we use the decomposition
\begin{equation*}
x=P_0x+Q_0x,
\end{equation*}
to get
\begin{eqnarray*}
(BP_0 x)^{\Delta}&=& (BP_0 B^-)^{\Delta}BP_0x+(BP_0 B^-)^{\Delta}BQ_0x\\ \\
&&+B^{\sigma}P_0^{\sigma}\left(G_1^{-1}\right)^{\sigma}C^{\sigma}P_0^{\sigma}x^{\sigma}+B^{\sigma}P_0^{\sigma}\left(G_1^{-1}\right)^{\sigma}C^{\sigma}Q_0^{\sigma}x^{\sigma}\\ \\
&&+ B^{\sigma}P_0^{\sigma}\left(G_1^{-1}\right)^{\sigma}f.
\end{eqnarray*}
Since $\text{im}\,Q_0=\text{ker}\,B$, we have
\begin{equation*}
(BP_0 B^-)^{\Delta}BQ_0 x=0.
\end{equation*}
Next,
\begin{eqnarray*}
E&=& G_1^{-1}(G_0+CQ_0)=G_1^{-1}G_0+G_1^{-1}CQ_0\\ \\
&=& E-Q_0+G_1^{-1}CQ_0,
\end{eqnarray*}
whereupon
\begin{equation}
\label{4.23} G_1^{-1}CQ_0=Q_0.
\end{equation}
Then
\begin{equation*}
B^{\sigma}P_0^{\sigma}\left(G_1^{-1}\right)^{\sigma}C^{\sigma}Q_0^{\sigma}= B^{\sigma}P_0^{\sigma}Q_0^{\sigma}=0
\end{equation*}
and we arrive at the equation
\begin{eqnarray*}
(BP_0 x)^{\Delta}&=& (BP_0 B^-)^{\Delta}BP_0x
+B^{\sigma}P_0^{\sigma}\left(G_1^{-1}\right)^{\sigma}C^{\sigma}P_0^{\sigma}x^{\sigma}+ B^{\sigma}P_0^{\sigma}\left(G_1^{-1}\right)^{\sigma}f.
\end{eqnarray*}
Since
\begin{equation*}
B^-BP_0=P_0P_0=P_0,
\end{equation*}
the last equation can be rewritten in the form
\begin{eqnarray*}
(BP_0 x)^{\Delta}&=& (BP_0 B^-)^{\Delta}BP_0x
+B^{\sigma}P_0^{\sigma}\left(G_1^{-1}\right)^{\sigma}C^{\sigma}B^{-\sigma}B^{\sigma}P_0^{\sigma}x^{\sigma}+ B^{\sigma}P_0^{\sigma}\left(G_1^{-1}\right)^{\sigma}f.
\end{eqnarray*}
We set
\begin{equation*}
u=BP_0x.
\end{equation*}
Then, we obtain the equation
\begin{equation}
\label{4.24} u^{\Delta}= (BP_0 B^-)^{\Delta}u
+B^{\sigma}P_0^{\sigma}\left(G_1^{-1}\right)^{\sigma}C^{\sigma}B^{-\sigma}u^{\sigma}+ B^{\sigma}P_0^{\sigma}\left(G_1^{-1}\right)^{\sigma}f.
\end{equation}
Now, we multiply both sides of \eqref{4.21} by $Q_0^{\sigma}$ and using that
\begin{equation*}
Q_0^{\sigma}(E-Q_0^{\sigma})=Q_0^{\sigma}-Q_0^{\sigma}Q_0^{\sigma}=Q_0^{\sigma}-Q_0^{\sigma}=0,
\end{equation*}
we find
\begin{equation*}
0= Q_0^{\sigma}\left(G_1^{-1}\right)^{\sigma}C^{\sigma}x^{\sigma}+Q_0^{\sigma}\left(G_1^{-1}\right)^{\sigma}f,
\end{equation*}
or
\begin{eqnarray*}
Q_0^{\sigma}\left(G_1^{-1}\right)^{\sigma}f&=& -Q_0^{\sigma}\left(G_1^{-1}\right)^{\sigma}C^{\sigma}x^{\sigma}\\ \\
&=& -Q_0^{\sigma}\left(G_1^{-1}\right)^{\sigma}C^{\sigma}(P_0^{\sigma}x^{\sigma}+Q_0^{\sigma}x^{\sigma})\\ \\
&=& -Q_0^{\sigma}\left(G_1^{-1}\right)^{\sigma}C^{\sigma}P_0^{\sigma}x^{\sigma}-Q_0^{\sigma}\left(G_1^{-1}\right)^{\sigma}C^{\sigma}Q_0^{\sigma}x^{\sigma}\\ \\
&=& -Q_0^{\sigma}\left(G_1^{-1}\right)^{\sigma}C^{\sigma}P_0^{\sigma}x^{\sigma}-Q_0^{\sigma}Q_0^{\sigma}x^{\sigma}\\ \\
&=& -Q_0^{\sigma}\left(G_1^{-1}\right)^{\sigma}C^{\sigma}P_0^{\sigma}x^{\sigma}-Q_0^{\sigma}x^{\sigma},
\end{eqnarray*}
where we have used \eqref{4.23}. We set
\begin{equation*}
v=Q_0x.
\end{equation*}
Then, we get
\begin{equation*}
v^{\sigma}= -Q_0^{\sigma}\left(G_1^{-1}\right)^{\sigma}C^{\sigma}u^{\sigma}-Q_0^{\sigma}\left(G_1^{-1}\right)^{\sigma}f.
\end{equation*}
By the last equation and \eqref{4.24}, we obtain the system
\begin{equation}
\label{4.25}
\begin{array}{lll}
u^{\Delta}&=& (BP_0 B^-)^{\Delta}u
+B^{\sigma}P_0^{\sigma}\left(G_1^{-1}\right)^{\sigma}C^{\sigma}B^{-\sigma}u^{\sigma}+ B^{\sigma}P_0^{\sigma}\left(G_1^{-1}\right)^{\sigma}f\\ \\
v^{\sigma}&=& -Q_0^{\sigma}\left(G_1^{-1}\right)^{\sigma}C^{\sigma}u^{\sigma}-Q_0^{\sigma}\left(G_1^{-1}\right)^{\sigma}f.
\end{array}
\end{equation}
By the first equation of \eqref{4.25} we find $u$ and then by the second equation we find $v^{\sigma}$. Then, for the solution $x$ of \eqref{1}, we have
\begin{equation*}
B^{-\sigma}u^{\sigma}+v^{\sigma}= B^{-\sigma}B^{\sigma}P_0^{\sigma}x^{\sigma}+Q_0^{\sigma} x^{\sigma}=P_0^{\sigma}x^{\sigma}+Q_0^{\sigma}x^{\sigma}=x^{\sigma}.
\end{equation*}
Here we write $B^{-\sigma}$ instead of $(B^-)^\sigma$. Now, we will show that the described process above is $\sigma$-reversible, i.e., if $u\in \mathcal{C}^1(I)$, $u\in \text{im}\,BP_0B^-$ and $v^{\sigma}\in \mathcal{C}(I)$ satisfy \eqref{4.25} and
\begin{equation}
\label{4.25.1} x^{\sigma}= B^{-\sigma}u^{\sigma}+v^{\sigma},
\end{equation}
then $x$ satisfies \eqref{1}. Really, since $u\in \text{im}\,BP_0B^-$, we have
\begin{equation*}
BP_0B^-u=u
\end{equation*}
and
\begin{equation*}
BB^-u= BB^-BP_0B^-u=BP_0P_0B^-u=BP_0 B^-u=u.
\end{equation*}
By the second equation of \eqref{4.25}, using that $Q_0^{\sigma}Q_0^{\sigma}=Q_0^{\sigma}$, we find
\begin{equation*}
v^{\sigma}=Q_0^{\sigma}v^{\sigma}.
\end{equation*}
Because $\text{im}\,Q_0= \text{ker}\,B$, we have
\begin{equation*}
B^{\sigma}v^{\sigma}= B^{\sigma}Q_0^{\sigma}v^{\sigma}=0.
\end{equation*}
Using \eqref{4.25.1}, we find
\begin{eqnarray*}
B^{\sigma}P_0^{\sigma}x^{\sigma}&=& B ^{\sigma}P_0^{\sigma}(B^{-\sigma}u^{\sigma}+v^{\sigma})= B^{\sigma}P_0^{\sigma}(B^{-\sigma}u^{\sigma}+Q_0^{\sigma}v^{\sigma})\\ \\
&=& B^{\sigma}P_0^{\sigma}B^{-\sigma}u^{\sigma}=u^{\sigma}.
\end{eqnarray*}
Hence,
\begin{equation*}
B^{-\sigma}u^{\sigma}= B^{-\sigma}B^{\sigma}P_0^{\sigma} x^{\sigma}= P_0^{\sigma}P_0^{\sigma} x^{\sigma}=P_0^{\sigma}x^{\sigma}=x^{\sigma}-Q_0^{\sigma}x^{\sigma}
\end{equation*}
using \eqref{4.25.1}, we find
\begin{equation*}
x^{\sigma}=B^{-\sigma}u^{\sigma}+v^{\sigma}= x^{\sigma}-Q_0^{\sigma}x^{\sigma}+v^{\sigma},
\end{equation*}
whereupon $v^{\sigma}=Q_0^{\sigma} x^{\sigma}$.
Next,
\begin{eqnarray*}
B^{\sigma}x^{\sigma}&=& B^{\sigma}(P_0^{\sigma}x^{\sigma}+Q_0^{\sigma}x^{\sigma})=B^{\sigma}(P_0^{\sigma}x^{\sigma}+ v^{\sigma})\\ \\
&=& B^{\sigma}P_0^{\sigma}x^{\sigma}+B^{\sigma}v^{\sigma}= u^{\sigma}.
\end{eqnarray*}
Note that the equation \eqref{4.24} is restated by the equation \eqref{4.20}. Then, multiplying \eqref{4.20} with $B^{\sigma}P_0^{\sigma}$, we find
\begin{equation*}
B^{\sigma}P_0^{\sigma}\left(G_1^{-1}\right)^{\sigma}G_0^{\sigma}B^{-\sigma}(Bx)^{\Delta}= B^{\sigma}P_0^{\sigma}\left(G_1^{-1}\right)^{\sigma}C^{\sigma}x^{\sigma}+B^{\sigma} P_0^{\sigma}\left(G_1^{-1}\right)^{\sigma}f,
\end{equation*}
which we premultiply by $B^{-\sigma}$ and using that $B^{-\sigma}B^{\sigma}P_0^{\sigma}=P_0^{\sigma}$, we arrive at
\begin{equation*}
P_0^{\sigma}\left(G_1^{-1}\right)^{\sigma}G_0^{\sigma}B^{-\sigma}(Bx)^{\Delta}= P_0^{\sigma}\left(G_1^{-1}\right)^{\sigma}C^{\sigma}x^{\sigma}+ P_0^{\sigma}\left(G_1^{-1}\right)^{\sigma}f,
\end{equation*}
from where
\begin{equation}
\label{4.27}P_0^{\sigma}\left(G_1^{-1}\right)^{\sigma}A^{\sigma}(Bx)^{\Delta}= P_0^{\sigma}\left(G_1^{-1}\right)^{\sigma}C^{\sigma}x^{\sigma}+ P_0^{\sigma}\left(G_1^{-1}\right)^{\sigma}f.
\end{equation}
Next, we multiply \eqref{4.20} by $Q_0^{\sigma}$ and we find
\begin{equation}
\label{4.28}Q_0^{\sigma}\left(G_1^{-1}\right)^{\sigma}A^{\sigma}(Bx)^{\Delta}= Q_0^{\sigma}\left(G_1^{-1}\right)^{\sigma}C^{\sigma}x^{\sigma}+ Q_0^{\sigma}\left(G_1^{-1}\right)^{\sigma}f.
\end{equation}
Now, we add  \eqref{4.27} and \eqref{4.28} and we find
\begin{equation*}
\left(G_1^{-1}\right)^{\sigma}A^{\sigma}(Bx)^{\Delta}= \left(G_1^{-1}\right)^{\sigma}C^{\sigma}x^{\sigma}+ \left(G_1^{-1}\right)^{\sigma}f,
\end{equation*}
whereupon we get \eqref{1}.
\begin{definition}
The equation \eqref{4.24} is said to be the inherent equation for the equation \eqref{1}.
\end{definition}
\begin{theorem}\label{main}
The subspace $\text{im}\,P_0$ is an invariant subspace for the equation \eqref{4.24}, i.e.,
\begin{equation*}
u(t_0)\in (\text{im}\,BP_0)(t_0)
\end{equation*}
for some $t_0\in I$ if and only if
\begin{equation*}
u(t)\in (\text{im}\,BP_0)(t)
\end{equation*}
for any $t\in I$.
\end{theorem}
\begin{proof}
Let $u\in \mathcal{C}^1(I)$ be a solution to the equation \eqref{4.24} so that
\begin{equation*}
(BP_0)(t_0)u(t_0)=u(t_0).
\end{equation*}
Hence,
\begin{eqnarray*}
u(t_0)&=& (BP_0)(t_0)u(t_0)= (BP_0P_0 P_0)(t_0)u(t_0)= (BP_0B^-B P_0)(t_0)u(t_0)\\ \\
&=& (BP_0B^-)(t_0) (BP_0)(t_0) u(t_0)= (BP_0B^-)(t_0)u(t_0).
\end{eqnarray*}
We multiply the equation \eqref{4.24} by $E-B^{\sigma}P_0^{\sigma}B^{-\sigma}$ and we get
\begin{eqnarray*}
(E-B^{\sigma}P_0^{\sigma}B^{-\sigma})u^{\Delta}&=&(E-B^{\sigma}P_0^{\sigma}B^{-\sigma})(BP_0B^-)^{\Delta}u\\ \\
&&+(E-B^{\sigma}P_0^{\sigma}B^{-\sigma})B^{\sigma}P_0^{\sigma}G_1^{-1\sigma}C^{\sigma}B^{-\sigma}u^{\sigma}\\ \\
&&+(E-B^{\sigma}P_0^{\sigma}B^{-\sigma})B^{\sigma}P_0^{\sigma}\left(G_1^{-1}\right)^{\sigma}f\\ \\
&=& (E-B^{\sigma}P_0^{\sigma}B^{-\sigma})(BP_0B^-)^{\Delta}u\\ \\
&&+(B^{\sigma}P_0^{\sigma}-B^{\sigma}P_0^{\sigma}B^{-\sigma}B^{\sigma}P_0^{\sigma})G_1^{-1\sigma}C^{\sigma}B^{-\sigma}u^{\sigma}\\ \\
&&+(B^{\sigma}P_0^{\sigma}-B^{\sigma}P_0^{\sigma}B^{-\sigma}B^{\sigma}P_0^{\sigma})G_1^{-1\sigma}f\\ \\
&=&  (E-B^{\sigma}P_0^{\sigma}B^{-\sigma})(BP_0B^-)^{\Delta}u\\ \\
&&+(B^{\sigma}P_0^{\sigma}-B^{\sigma}P_0^{\sigma})G_1^{-1\sigma}C^{\sigma}B^{-\sigma}u^{\sigma}\\ \\
&&+(B^{\sigma}P_0^{\sigma}-B^{\sigma}P_0^{\sigma})G_1^{-1\sigma}f\\ \\
&=& (E-B^{\sigma}P_0^{\sigma}B^{-\sigma})(BP_0B^-)^{\Delta}u,
\end{eqnarray*}
i.e.,
\begin{equation}
\label{4.29} (E-B^{\sigma}P_0^{\sigma}B^{-\sigma})u^{\Delta}=(E-B^{\sigma}P_0^{\sigma}B^{-\sigma})(BP_0B^-)^{\Delta}u.
\end{equation}
Set
\begin{equation*}
v=(E- BP_0B^-)u.
\end{equation*}
Then
\begin{eqnarray*}
v^{\Delta}&=& (E-B^{\sigma}P_0^{\sigma}B^{-\sigma})u^{\Delta}+(E-BP_0B^-)^{\Delta}u\\ \\
&=& (E-B^{\sigma}P_0^{\sigma}B^{-\sigma})(BP_0B^-)^{\Delta}u+(E-BP_0B^-)^{\Delta}u\\ \\
&=& \left((E-BP_0B^-)BP_0B^-\right)^{\Delta}u\\ \\
&&-(E-BP_0B^-)^{\Delta}BP_0B^-u\\ \\
&&+ (E-BP_0B^-)^{\Delta}u\\ \\
&=& (BP_0B^--BP_0B^-BP_0B^-)^{\Delta} u\\ \\
&&+(E-BP_0B^-)^{\Delta}(E-BP_0B^-)u\\ \\
&=& (E-BP_0B^-)^{\Delta}v,
\end{eqnarray*}
i.e.,
\begin{equation*}
v^{\Delta}=(E-BP_0B^-)^{\Delta}v.
\end{equation*}
Note that
\begin{eqnarray*}
v(t_0)&=& u(t_0)-(BP_0B^-)(t_0)u(t_0)\\ \\
&=& u(t_0)-u(t_0)\\ \\
&=& 0.
\end{eqnarray*}
Thus, we get the following IVP
\begin{eqnarray*}
v^{\Delta}&=&(E-BP_0B^-)^{\Delta}v\quad \text{on}\quad I\\ \\
v(t_0)&=& 0.
\end{eqnarray*}
Therefore $v=0$ on $I$ and then
\begin{equation*}
BP_0B^-u=u\quad \text{on}\quad I.
\end{equation*}
Hence, using that
\begin{equation*}
\text{im}\,BP_0=\text{im}\,BP_0B^-,
\end{equation*}
we get
\begin{eqnarray*}
BP_0u&=& u\quad \text{on}\quad I.
\end{eqnarray*}
This completes the proof.
\end{proof}
\section{Examples}
\begin{example}
Let $\mathbb{T}=\mathbb{N}_0$
Consider \eqref{4.3} for
\begin{eqnarray*}
A(t)&=& \left(
          \begin{array}{ccc}
            -1 & t+1 & -1 \\
            0 & 0 & 0 \\
            0 & 2t+2 & -1 \\
          \end{array}
        \right),\quad 
C(t)= \left(
          \begin{array}{ccc}
            0 & 0 & 1 \\
            0 & -t & 1 \\
            0 & 2 & 1 \\
          \end{array}
        \right),\\ \\
P(t)&=& \left(
          \begin{array}{ccc}
            1 & 0 & 0 \\
            0 & 0 & 0 \\
            0 & -(t+1) & 1 \\
          \end{array}
        \right),\quad t\in \mathbb{T}.
\end{eqnarray*}
We have $\sigma(t)=t+1$, $t\in \mathbb{T}$, and
$P(t)P(t)=P(t)$ , $A(t) P(t)=A(t)$, $t\in \mathbb{T}$.
Then
\begin{eqnarray*}
Q(t)&=& I-P(t)= \left(
      \begin{array}{ccc}
        0 & 0 & 0 \\
        0 & 1 & 0 \\
        0 & t+1 & 0 \\
      \end{array}
    \right)\\ \\
A_1(t)&=& A(t)+C(t)Q(t)= \left(
      \begin{array}{ccc}
        -1 & 2(t+1) & -1 \\
        0 & 1 & 0 \\
        0 & 3t+5 & -1 \\
      \end{array}
    \right),\\ \\
A_1^{-1}(t)&=&  \left(
      \begin{array}{ccc}
        -1 & -(t+3) &1 \\
        0 & 1 & 0 \\
        0 & 3t+5 & -1 \\
      \end{array}
    \right),\quad t\in \mathbb{T}.
\end{eqnarray*}
Hence,
\begin{eqnarray*}
    \left(A_1^{-1}\right)^{\sigma}(t)&=& \left(
      \begin{array}{ccc}
        -1 & -(t+4) & 1 \\
        0 &  1 & 0 \\
        0 & 3t+8 & -1 \\
      \end{array}
    \right),\\ \\ 
C^{\sigma}(t)&=& \left(
      \begin{array}{ccc}
        0 & 0 & 1 \\
        0 & -t-1 & 1 \\
        0 & 2 & 1 \\
      \end{array}
    \right),\quad 
P^{\sigma}(t)= \left(
      \begin{array}{ccc}
        1 & 0 & 0 \\
        0 & 0 & 0 \\
        0 & -(t+2) & 1 \\
      \end{array}
    \right),\\ \\ 
Q^{\sigma}(t)&=& \left(
      \begin{array}{ccc}
        0 & 0 & 0 \\
        0 & 1 & 0 \\
        0 & t+2 & 0 \\
      \end{array}
    \right),\quad 
P^{\Delta}(t)= \left(
      \begin{array}{ccc}
        0 & 0 & 0 \\
        0 & 0 & 0 \\
        0 & -1 & 0 \\
      \end{array}
    \right),\quad t\in \mathbb{T},
\end{eqnarray*}
and
\begin{eqnarray*}
P^{\sigma}(t)\left(A_1^{-1}\right)^{\sigma}(t)
&=& \left(
             \begin{array}{ccc}
               -1 &-(t+4) & 1 \\
               0 &  0 & 0 \\
               0 & 2(t+3) & -1 \\
             \end{array}
           \right),\\ \\
P^{\sigma}(t)\left(A_1^{-1}\right)^{\sigma}(t)C^{\sigma}(t)
&=& \left(
      \begin{array}{ccc}
        0 & (t+2)(t+3) & -(t+4) \\
        0 &  0 &  0 \\
        0 & -2(t+2)^2 & 2t+5 \\
      \end{array}
    \right),\\ \\
Q^{\sigma}(t)\left(A_1^{-1}\right)^{\sigma}(t)
&=& \left(
      \begin{array}{ccc}
        0 & 0 & 0 \\
        0 & 1 & 0 \\
        0 & t+2 & 0 \\
      \end{array}
    \right),\\ \\
Q^{\sigma}(t) \left(A_1^{-1}\right)^{\sigma}(t)C^{\sigma}(t)
&=& \left(
      \begin{array}{ccc}
        0 & 0 & 0 \\
        0 & -t-1 & 1 \\
        0 & -(t+1)(t+2) & t+2 \\
      \end{array}
    \right),\quad t\in \mathbb{T}.
\end{eqnarray*}
Then the system \eqref{4.10} takes the form
\begin{eqnarray*}
u^{\Delta}(t)&=& \left(
             \begin{array}{ccc}
               0 & 0 & 0 \\
               0 & 0 & 0 \\
               0 & -1 & 0 \\
             \end{array}
           \right)u(t)+\left(
                            \begin{array}{ccc}
                              0 & (t+2)(t+3) & -(t+4) \\
                              0 & 0 & 0 \\
                              0 & -2(t+2)^2 & 2t+5 \\
                            \end{array}
                          \right)u^{\sigma}(t)+ \left(
      \begin{array}{ccc}
        -1 & -(t+4) & -1 \\
        0 &  0 &  0 \\
        0 & 2(t+3) & -1 \\
      \end{array}
    \right)f(t)\\ \\
v^{\sigma}(t)&=& -\left(
              \begin{array}{ccc}
                0 & 0 & 0 \\
                0 & -t-1 & 1 \\
                0 & -(t+1)(t+2) & t+2 \\
              \end{array}
            \right)u^{\sigma}(t)-\left(
                             \begin{array}{ccc}
                               0 & 0 & 0 \\
                               0 & 1 & 0\\
                               0 & t+2 & 0 \\
                             \end{array}
                           \right)f(t),\quad t\in \mathbb{T}.
\end{eqnarray*}
\end{example}
\begin{example}
Let $\mathbb{T}=2^{\mathbb{N}_0}$. Consider \eqref{1} for 
\begin{eqnarray*}
A(t)&=& \left(
          \begin{array}{ccc}
            t & 0 & 0 \\
            0 & 1 & 0 \\
            0 & 0 & t^2 \\
            0 & 0 & 0 \\
            0 & 0 & 0 \\
          \end{array}
        \right),\quad 
B(t)= \left(
        \begin{array}{ccccc}
          t & 0 & 0 & 0 & 0 \\
          0 & t^2 & 0 & 0 & 0 \\
          0 & 0 & 1 & 0 & 0 \\
        \end{array}
      \right),\\ \\
C(t)&=& \left(
          \begin{array}{ccccc}
            0 & 0 & 0 & -1 & 1 \\
            0 & 0 & t & 1 & 0 \\
            0 & -1 & 0 & 0 & 0 \\
            -1 & 1 & 0 & -t^2 & 0 \\
            1 & 0 & 0 & 0 & t^2 \\
          \end{array}
        \right),\quad t\in \mathbb{T}.
\end{eqnarray*}
Here $\sigma(t)=2t$, $t\in \mathbb{T}$, and
\begin{eqnarray*}
R(t)&=& \left(
      \begin{array}{ccc}
        1 & 0 & 0 \\
        0 & 1 & 0 \\
        0 & 0 & 1 \\
      \end{array}
    \right), \quad G_0(t)= A(t) B(t)
= \left(
                 \begin{array}{ccccc}
                   t^2 & 0 & 0 & 0 & 0 \\
                   0 & t^2 & 0 & 0 & 0 \\
                   0 & 0 & t^2 & 0 & 0 \\
                   0 & 0 & 0 & 0 & 0 \\
                   0 & 0 & 0 & 0 & 0 \\
                 \end{array}
               \right),\\ \\
Q_0(t)&=& \left(
         \begin{array}{ccccc}
           0 & 0 & 0 & 0 & 0 \\
           0 & 0 & 0 & 0 & 0 \\
           0 & 0 & 0 & 0 & 0 \\
           0 & 0 & 0 & 1 & 0 \\
           0 & 0 & 0 & 0 & 1 \\
         \end{array}
       \right),\quad P_0(t)= \left(
         \begin{array}{ccccc}
           1 & 0 & 0 & 0 & 0 \\
           0 & 1 & 0 & 0 & 0 \\
           0 & 0 & 1 & 0 & 0 \\
           0 & 0 & 0 & 0 & 0 \\
           0 & 0 & 0 & 0 & 0 \\
         \end{array}
       \right),\quad t\in \mathbb{T}.
\end{eqnarray*}
Hence,
\begin{eqnarray*}
G_1(t)&=& G_0(t)+C(t)Q_0(t)
= \left(
      \begin{array}{ccccc}
        t^2 & 0 & 0 & -1 & 1 \\
        0 & t^2 & 0 & 1 & 0 \\
        0 & 0 & t^2 & 0 & 0 \\
        0 & 0 & 0 & -t^2 & 0 \\
        0 & 0 & 0 & 0 & t^2 \\
      \end{array}
    \right),\\ \\
B^-(t)&=& \left(
         \begin{array}{ccc}
           \frac{1}{t} & 0 & 0 \\
           0 & \frac{1}{t^2} & 0 \\
           0 & 0 & 1 \\
           0 & 0 & 0 \\
           0 & 0 & 0 \\
         \end{array}
       \right),\quad t\in \mathbb{T}.
\end{eqnarray*}
Next,
\begin{eqnarray*}
B^{-\sigma}(t)&=& \left(
                    \begin{array}{ccc}
                      \frac{1}{2t} & 0 & 0 \\
                      0 & \frac{1}{4t^2} & 0 \\
                      0 & 0 & 1 \\
                      0 & 0 & 0 \\
                      0 & 0 & 0 \\
                    \end{array}
                  \right),\quad 
B^{\sigma}(t)= \left(
                   \begin{array}{ccccc}
                     2t & 0 & 0 & 0 & 0 \\
                     0 & 4t^2 & 0 & 0 & 0 \\
                     0 & 0 & 1 & 0 & 0 \\
                   \end{array}
                 \right),\\ \\
C^{\sigma}(t)&=& \left(
                   \begin{array}{ccccc}
                     0 & 0 & 0 & -1 & 1 \\
                     0 & 0 & 2t & 1 & 0 \\
                     0 & -1 & 0 & 0 & 0 \\
                     -1 & 1 & 0 & -4t^2 & 0 \\
                     1 & 0 & 0 & 0 & 4t^2 \\
                   \end{array}
                 \right),\quad 
Q^{\sigma}_0(t)= \left(
                     \begin{array}{ccccc}
                       0 & 0 & 0 & 0 & 0 \\
                       0 & 0 & 0 & 0 & 0 \\
                       0 & 0 & 0 & 0 & 0 \\
                       0 & 0 & 0 & 1 & 0 \\
                       0 & 0 & 0 & 0 & 1 \\
                     \end{array}
                   \right),\\ \\
P_0^{\sigma}(t)&=& \left(
                     \begin{array}{ccccc}
                       1 & 0 & 0 & 0 & 0 \\
                       0 & 1 & 0 & 0 & 0 \\
                       0 & 0 & 1 & 0 & 0 \\
                       0 & 0 & 0 & 0 & 0 \\
                       0 & 0 & 0 & 0 & 0 \\
                     \end{array}
                   \right),\quad t\in \mathbb{T}.
\end{eqnarray*}
Note that
\begin{eqnarray*}
\det G_1(t)&=& -t^{10}\ne 0,\quad t\in \mathbb{T}, 
\end{eqnarray*}
and
\begin{eqnarray*}
G_1^{-1}(t)&=&  \left(
      \begin{array}{ccccc}
        \frac{1}{t^2} & 0 & 0 & -\frac{1}{t^4} & -\frac{1}{t^4} \\
        0 & \frac{1}{t^2} & 0 & \frac{1}{t^4} & 0 \\
        0 & 0 & \frac{1}{t^2} & 0 & 0 \\
        0 & 0 & 0 & -\frac{1}{t^2} & 0 \\
        0 & 0 & 0 & 0 & \frac{1}{t^2} \\
      \end{array}
    \right),\\ \\
\left(G^{-1}\right)^{\sigma}(t)&=&\left(
      \begin{array}{ccccc}
        \frac{1}{4t^2} & 0 & 0 & -\frac{1}{16t^4} & -\frac{1}{16t^4} \\
        0 & \frac{1}{4t^2} & 0 & \frac{1}{16t^4} & 0 \\
        0 & 0 & \frac{1}{4t^2} & 0 & 0 \\
        0 & 0 & 0 & -\frac{1}{4t^2} & 0 \\
        0 & 0 & 0 & 0 & \frac{1}{4t^2} \\
      \end{array}
    \right),\quad t\in \mathbb{T}.
\end{eqnarray*}
Therefore
\begin{eqnarray*}
B(t) P_0(t) B^-(t)
&=&\left(
     \begin{array}{ccc}
       1 & 0 & 0 \\
       0 & 1 & 0 \\
       0 & 0 & 1 \\
     \end{array}
   \right),\quad
(BP_0B^-)^{\Delta}(t)=0,\\ \\
B^{\sigma}(t) P_0^{\sigma}(t) \left(G_1^{-1}\right)^{\sigma}(t)&=& \left(
      \begin{array}{ccccc}
        \frac{1}{2t} & 0 & 0 & -\frac{1}{8t^3} & -\frac{1}{8t^3} \\
        0 & 1 & 0 & \frac{1}{4t^2} & 0 \\
        0 & 0 & \frac{1}{4t^2} & 0 & 0 \\
      \end{array}
    \right),
\end{eqnarray*}
\begin{eqnarray*}
    B^{\sigma}(t) P_0^{\sigma}(t) \left(G_1^{-1}\right)^{\sigma}(t)C^{\sigma}(t)B^{-\sigma}(t)
    &=& \left(
          \begin{array}{ccc}
            0 & -\frac{1}{32t^5} & 0 \\
            -\frac{1}{8t^3} & \frac{1}{16t^4} & 2t \\
            0 & -\frac{1}{16t^4} & 0 \\
          \end{array}
        \right),\\ \\
Q_0^{\sigma}(t) \left(G_1^{-1}\right)^{\sigma}(t)&=& \left(
      \begin{array}{ccccc}
        0 & 0 & 0 & 0 & 0 \\
        0 & 0 & 0 & 0 & 0 \\
        0 & 0 & 0 & 0 & 0 \\
        0 & 0 & 0 & -\frac{1}{4t^2} & 0 \\
        0 & 0 & 0 & 0 & \frac{1}{4t^2} \\
      \end{array}
    \right),\\ \\
Q_0^{\sigma}(t) \left(G_1^{-1}\right)^{\sigma}(t) C^{\sigma}(t)B^{-\sigma}(t)
&=& \left(
      \begin{array}{ccc}
        0 & 0 & 0 \\
        0 & 0 & 0 \\
        0 & 0 & 0 \\
        \frac{1}{8t^3} & -\frac{1}{16t^4} & 0 \\
        \frac{1}{8t^3} & 0 & 0 \\
      \end{array}
    \right),\quad t\in \mathbb{T}.
\end{eqnarray*}
Then, the decoupling \eqref{4.24} takes the form
\begin{eqnarray*}
u^{\Delta}(t)&=& \left(
             \begin{array}{ccc}
               0 & -\frac{1}{32t^5} & 0 \\
               -\frac{1}{8t^3} & \frac{1}{16t^4} & 2t \\
               0 & -\frac{1}{16t^4} & 0 \\
             \end{array}
           \right)u^{\sigma}(t)+\left(
      \begin{array}{ccccc}
        \frac{1}{2t} & 0 & 0 & -\frac{1}{8t^3} & -\frac{1}{8t^3} \\
        0 & 1 & 0 & \frac{1}{4t^2} & 0 \\
        0 & 0 & \frac{1}{4t^2} & 0 & 0 \\
      \end{array}
    \right)f(t)\\ \\
v^{\sigma}(t)&=& -\left(
      \begin{array}{ccc}
        0 & 0 & 0 \\
        0 & 0 & 0 \\
        0 & 0 & 0 \\
        \frac{1}{8t^3} & -\frac{1}{16t^4} & 0 \\
        \frac{1}{8t^3} & 0 & 0 \\
      \end{array}
    \right)u^{\sigma}(t)-\left(
                             \begin{array}{ccccc}
                               0 & 0 & 0 & 0 & 0 \\
                               0 & 0 & 0 & 0 & 0 \\
                               0 & 0 & 0 & 0 & 0 \\
                               0 & 0 & 0 & -\frac{1}{4t^2} & 0 \\
                               0 & 0 & 0 & 0 & \frac{1}{4t^2} \\
                             \end{array}
                           \right)f(t),\quad t\in \mathbb{T}.
\end{eqnarray*}
\end{example}

\section{Conclusion} The problem of decoupling is very significant in the theory of dynamical systems of any origin. First of all, it simplifies the process of finding the analytic solution and qualitative studies. For instance, it is widely applied in the Stability Theory (e.g., for hyperbolic, partially hyperbolic, non-uniformly hyperbolic, or regular systems).

The case of an arbitrary time scale is more sophisticated than that of classical ordinary differential equations for the following reasons:
\begin{enumerate}
\item time scale calculus involves more sophisticated formulas than the classical ones and the classical theory cannot be immediately translated to the language of time scales;
\item the theory of time scale systems is much less developed than that of ordinary differential equations;
\item autonomous systems are much more difficult to study for generic time scales.
\end{enumerate}

The main result of the paper is Theorem \ref{main}. Our proofs are constructive.

Although the approach of this paper is quite straightforward, special techniques, based projector approach were elaborated. The obtained results are illustrated with  specific examples.

\section*{Acknowledgements}
The work of the second co-author was supported by Gda\'{n}sk University of Technology by the DEC 14/2021/IDUB/I.1 grant under the Nobelium - 'Excellence Initiative - Research University' program.

\end{document}